\documentclass[12pt]{article}

\usepackage[english]{babel}
\usepackage{amsfonts}
\usepackage{amsmath}
\usepackage{amssymb}
\usepackage{amsthm}
\usepackage{xcolor}
\usepackage{esint} %% average integration package
\usepackage[utf8]{inputenc}
\usepackage{graphicx}
\usepackage{lipsum}
\newcommand\blfootnote[1]{%
  \begingroup
  \renewcommand\thefootnote{}\footnote{#1}%
  \addtocounter{footnote}{-1}%
  \endgroup
}

\usepackage[letterpaper,top=2cm,bottom=2cm,left=3cm,right=3cm,marginparwidth=1.75cm]{geometry}

\newtheorem{theorem}{Theorem}[section]
\newtheorem*{theorem*}{Theorem}

\newtheorem{lemma}[theorem]{Lemma}
\newtheorem{corollary}[theorem]{Corollary}
\theoremstyle{definition}
\newtheorem{definition}[theorem]{Definition}

\DeclareMathOperator{\Ric}{Ric}

\DeclareMathOperator{\tr}{tr}

\def \RR {\mathbb{R}}

\def \nablas {\nabla^\Sigma}

\title{The Log-Sobolev inequality for a submanifold in manifolds with asymptotic non-negative intermediate Ricci curvature }
\author{Jihye Lee, Fabio Ricci}
\date{}

\begin{document}
%\tableofcontents
\maketitle

\begin{abstract}
    We prove a sharp Log-Sobolev inequality for submanifolds of a complete non-compact Riemannian manifold with asymptotic non-negative intermediate Ricci curvature and Euclidean volume growth. Our work extends a result of Dong-Lin-Lu \cite{DongLinLu2} which already generalizes Yi-Zheng's \cite{Yi-Zheng} and Brendle's \cite{B_log}.  %We don't add any restriction on the dimension/codimension. The inequality includes explicitly the asymptotic volume ratio of the ambient manifold and a term involving the Mean curvature, that is why our result can also be seen as a Micheal-Simon-Sobolev type inequality. As a corollary we weaken some conditions on the non-existence of closed minimal submanifolds.
\end{abstract}
\blfootnote{Jihye Lee: UC Santa Barbara, Department of Mathematics, email: jihye@ucsb.edu. }
\blfootnote{Fabio Ricci: UC Santa Barbara, Department of Mathematics, email: FabioRicci@ucsb.edu. }
\section{Introduction}
Log-Sobolev inequalities have been already studied in mathematical physics and information theory in the 50's (\cite{Stam}, \cite{Federbush}) then it has been extensively studied in analysis and geometry (see e.g. \cite{Gross2}, \cite{Gross}, \cite{Simon}, \cite{Ledoux}, \cite{Ecker}, \cite{Villani}, \cite{Ohta}, \cite{Perelman}), with recent developments including their applications to heat kernel and Hamilton-Jacobi equations, inequalities in convex geometry, and their role in optimal transport.
Brendle \cite{B_log} obtained a sharp log-Sobolev inequality for submanifolds in Euclidean space, improving \cite{Ecker}, using the Alexandroff-Bakelman-Pucci (ABP) method and optimal transport. Later Yi and Zheng  \cite{Yi-Zheng} extended this result to Riemannian manifolds with non-negative sectional curvature and Euclidean volume growth conditions. Subsequently Dong-Lin-Lu \cite{DongLinLu2} further extended this to non-negative asymptotic sectional curvature.
In this paper, we prove that the logarithmic Sobolev inequality holds 
under non-negative intermediate Ricci curvature. Furthermore, we extend our result  to the asymptotic non-negative intermediate Ricci curvature by introducing error terms that depends only on the curvature's decay at infinity, as in \cite{DongLinLu2}.

Let $(M^n,g)$ be a complete non-compact Riemannian manifold. For $1 \leq k\leq n-1$ we consider a $k$-dimensional subspace 
$P$ of a tangent space $T_xM$ at $x\in M$.
Given any tangent vector $v \in T_xM$ with $v \bot P$,
the $k$-intermediate Ricci curvature ($k$th-Ricci curvature) with respect to $P$ in the direction of $v$ is defined as
$$\Ric_k (P,v) = \sum_{i=1}^k  \mathrm{Sec}(v, e_i) |v|^2,$$
where $\{e_1, \ldots, e_k\}$ is an orthonormal basis of $P$.
The $k$-intermediate Ricci curvature  interpolates between Ricci and sectional curvature. 
A Riemannian manifold has non-negative $k$-Ricci curvature if
 $\Ric_k(P,v) \geq 0$ for any $x\in M$, $k$-dimensional subspace $P\subset T_xM$, and a unit tangent vector $v\in T_xM$ perpendicular to $P$ (note that some papers \cite{wang22, Mondino} require  the stronger condition that this holds for all $v$, not necessary just the perpendicular ones). This condition is denoted by $\Ric_k \geq 0$.
In particular, it exhibits a monotonicity property: if $n \leq m$, then $\Ric_n \geq 0$ implies $\Ric_m \geq 0$. 
This curvature condition has been well studied to explore the gap of the global results with sectional curvature bounds and Ricci curvature bounds, see for example: \cite{shen, Wu,Chahine,Gui2,Gui3,Reiser,Mondino,ma-wu,wang22}. Our work follows this spirit and generalizes Yi-Zheng's result with sectional curvature lower bound (Theorem 1.1, \cite{Yi-Zheng}) and the corresponding asymptotic extension by Dong-Lin-Lu (Theorem 1.1, \cite{DongLinLu2}). Ketterer-Mondino \cite{Mondino} noted that it is possible to characterize lower bounds of $k$-intermediate Ricci curvature via optimal transport. We will adopt this approach, developed in \cite{ma-wu} and \cite{wang22}, to prove our main result.

Let $M$ be an $n$-dimensional non-compact Riemmanian manifold with non-negative Ricci curvature. 
The asymptotic volume ratio of $M$ is defined as
$$\theta := \lim_{r \to \infty}\frac{|B_r(p)|}{\omega_nr^n},$$
where $p$ is some fixed point in $M$, $\omega_n$ is the volume of the unit ball in Euclidean space $\RR^n$, and $|B_r(p)|$ is the volume of a ball of radius $r$ centered at $p$ in $M$. The Bishop-Gromov volume comparison assures that the limit exists, it does not depend on the choice of $p$ and that $\theta \leq 1$. We say that $M$ has Euclidean volume growth if $\theta >0$.

In this paper we are using the definition of asymptotically non negative sectional ( resp. Ricci) curvature given in Abresch \cite{Abresch} (see also \cite{Abresch2}, \cite{Kasue}, \cite{Zhang}). Zhu (see theorem 2.1 in \cite{Zhu}) proved the equivalent of the classical Bishop-Gromov volume comparison Theorem with $\operatorname{Ric}\geq 0$ by replacing $\mathbb{R}^n$ with a different model space.  We will need to define the equivalent of the usual asymptotic volume ratio. 

\begin{definition}[Abresch \cite{Abresch}]
An $n$-dimensional complete non-compact Riemannian manifold $(M,g)$ with base point $o$ has asymptotically non-negative Sectional curvature (Ricci curvature, respectively) if and only if there exists any non-negative, non-increasing function  $\lambda: [0,\infty)\rightarrow[0,\infty) $ such that the following holds:
\begin{itemize}
    \item[(1)] $b_0(\lambda)=\int_0^\infty t\lambda(t)dt < \infty$
    \item[(2)] $\operatorname{Sec}\geq -\lambda(d(o,p)) $ at each point $p\in M$. \text{   }  $\left( \operatorname{Ric}\geq -(n-1)\lambda(d(o,p)), \text{ respectively} \right)$
\end{itemize}
\end{definition}
The first condition also implies that $b_1(\lambda)=\int_0^\infty \lambda(t)dt < \infty$.
To extend this notion of asymptotically non-negative sectional curvature to intermediate $k$-Ricci curvature, one can simply replace the second  condition with:
\begin{align*}
    \operatorname{Ric}_k\geq -k\lambda(d(o,p)) \text{ at each point } p \in M.
\end{align*}
More precisely, $\Ric_k(P,v) \geq - k \lambda (d(o,p))$ for any $k$ dimensional subspace $P \subset T_pM$ and unit tangent vector $v$ perpendicular to $P$.

The usual non-negative sectional curvature (Ricci or $\Ric_k$ respectively) condition is equivalent to requiring $\lambda \equiv 0$.
It's immediately evident from the definition to see that this class of manifolds includes those with non-negative sectional curvature (Ricci or $\operatorname{Ric}_k$ respectively) outside a compact set and asymptotically flat manifolds as well. Furthermore, the monotonicity of the intermediate curvature  still holds:  if $k_1 \leq k_2$, then
    $\Ric_{k_1} \geq - k_1\lambda ( d(o,p))$ implies $\Ric_{k_2} \geq - k_2 \lambda ( d(o,p))$.

In this setting, we also replace $\theta$, the asymptotic volume ratio of $M$, with another similar quantity that will keep track of the geometry at infinity.  To achieve this, we define $h(t)$ as the unique solution of the following ODE:
\begin{align*}
   \begin{cases}
h''(t) = \lambda(t) h(t),\\
h(0) = 0, h'(0)=1.
\end{cases}
\end{align*}
We now define the asymptotic volume ratio of $M$ with respect to $h$ by
\begin{align*}
    \theta_h=\lim_{r\rightarrow\infty}\frac{\mathcal|B_r(o)|}{n\omega_n\int_0^rh^{n-1}(t)dt},
\end{align*}
In the above definition, $|B_r(o)|$ represents the volume of the ball of radius $r$ centered at $o$ in $M$ and $\omega_n$ represents the volume of the unit ball in $\RR^n$. In \cite{Zhu}, to prove the corresponding version of the Bishop-Gromov volume comparison in this setting, Zhu noted that the following function of $r$ is non-increasing:
\begin{align*}
    \frac{|B_r(o)|}{n\omega_n\int_0^rh^{n-1}(t)dt}.
\end{align*}
This ensures that the limit exists and $\theta_h$ is well-defined. Similarly, we say that $M$ has Euclidean volume growth if $\theta_h>0$.
In particular, when a manifold has non-negative intermediate Ricci curvature,  $\lambda(t) \equiv 0$, $h(t) = t$, and $\theta_h$ is the usual asymptotic volume ratio $\theta$.
In the second section of this paper we will show a connection between $\theta_h$ and an integral of a certain Guassian function. To state our theorem, we first define, for any non-negative $t$, a decreasing function $P(t)$:
\begin{align*}
    P(t)=(4 \pi )^{-\frac{n}{2}}\int_{\mathbb{R}^N} e^{-\frac{(|x|+t)^2}{4}}dx
\end{align*}
Notice that $P(0)=1$.
Below is our main result.
\begin{theorem}\label{thm:main-logASYMP}
Let $M^{n+m}$ be a complete, non-compact manifold of dimension $n+m$ with  asymptotically non-negative $\mathrm{Ric}_{k}$, where $k=\min\{n-1,m-1\}$ and Euclidean volume growth. Assume that $\Sigma^n$ is an $n$-dimensional compact submanifold without a boundary.
Then, for any positive smooth function $f$, we have
\begin{align}\label{ineq:log-sobolevASYMP}
    \int_\Sigma &f \left(\log f + \log P(4b_1 n ) \theta_h + \frac{n}{2}\log (4 \pi) +n  + 4b_1^2 n^2 + (n+m-1) \log \frac{1+b_0}{e^{2r_0b_1 + b_0}}\right)d V (x) \nonumber\\
    &- \int_\Sigma \frac{|\nablas f|^2}{f} - \int_\Sigma f |H|^2  \leq \int_\Sigma f d V (x) \left(\log \int_\Sigma f(x) dV (x)\right),
\end{align}
where $\theta_h$ is the asymptotic volume ratio of $M$ with respect to $h$ and $H$ is the mean curvature vector of $\Sigma$ and $\alpha=(n+m-1)(2r_0b_1+b_0)$, where $r_0=\operatorname{max}_{x\in \Sigma }d(o,x)$.
\end{theorem}
Similar to the Michael-Simon-Sobolev inequality, the inequality \eqref{ineq:log-sobolevASYMP} contains a term that involves the mean curvature and depends on the asymptotic volume ratio of $M$.
If $\lambda\equiv 0$, the usual non-negative $\operatorname{Ric}_k$ condition, in other words $b_0=b_1=0$, then our Theorem \ref{thm:main-logASYMP} becomes:
\begin{corollary}\label{coro:intermediate}
Let $M^{n+m}$ be a complete, non-compact manifold of dimension $n+m$ with Euclidean volume growth and $\mathrm{Ric}_{k} \geq 0$, where $k=\min\{n-1,m-1\}$. Let $\Sigma^n$ be a $n$-dimensional compact submanifold without a boundary.
Then, for any positive smooth function $f$, we have

\begin{align}\label{MAIN}
    \int_\Sigma &f\left(\log f + n+\frac{n}{2} \log (4\pi) + \log \theta \right) dV - \int_\Sigma \frac{|\nablas f|^2}{f} dV - \int_\Sigma f|H|^2 dV \nonumber\\
    &\leq \left(\int_\Sigma f dV\right)\log \left(\int_\Sigma f dV\right),
\end{align}
where $\theta$ is the asymptotic volume ratio of $M$ and $H$ is the Mean curvature vector of $\Sigma$.
\end{corollary}

This gives the following two consequences by choosing $f\equiv1$ in \eqref{MAIN} and apply the argument in Section 4 of \cite{Yi-Zheng}.
\begin{corollary}\label{coro:1-2}
    Let $M^{n+m}$ be a complete, non-compact manifold of dimension $n+m$ with non-negative intermediate $k$-Ricci curvature, where $k = \min\{n-1,m-1\}$. Suppose $M$ has a Euclidean volume growth. Then, there is no closed minimal submanifold of dimension $n$ in $M$.
\end{corollary}

\begin{corollary}
    Let $M^{n+m}$ be a complete, non-compact manifold of dimension $n+m$ with non-negative intermediate $k$-Ricci curvature, where $k = \min\{n-1,m-1\}$. Suppose there is a closed minimal submanifold of dimension $n$ in $M$. Then, the asymptotic volume ratio of $M$ is zero.
\end{corollary}
In the case of any hypersurface ($m=1$) under non-negative Ricci curvature condition, these two corollaries have been proved by Agostiniani-Fogagnolo-Mazzieri (see Theorem 1.6 in \cite{mazzieri}). We do not recover this case in our result since $k= \min\{n-1,m-1\}$ would be zero. On the other hand, for higher dimension and codimension, these corollaries are also new in the usual non-negative $k$-intermediate Ricci curvature setting. When $k=1$, the standard $\operatorname{Ric}_1\geq 0$ is equivalent to the non-negative sectional curvature condition. Thus, we recover the corollaries in Yi-Zheng (Corollaries 1.1 and 1.2 in \cite{Yi-Zheng}).

The same argument (Section 4, \cite{Yi-Zheng}) cannot be applied in the case of asymptotic non-negative $k$-intermediate Ricci curvature setting because the error term we have introduced is not scale-invariant and will blow up if we blow down the metric as in \cite{Yi-Zheng}.

Note that the previous two corollaries do not hold for $\Ric\geq 0$ total dimension $4$; a counterexample is the Eguchi-Hanson metric on $TS^2$, which is Ricci flat with Euclidean volume growth and has a totally geodesic submanifold $S^2$. See, for example, \cite[Page 270]{Anderson}.

The proof of Theorem \ref{thm:main-logASYMP} relies on the Alexandroff-Bakelman-Pucci (ABP) method, a standard technique in Euclidean space \cite{Trud}. Cabrè \cite{Cabre} extended this method to manifolds. The classical estimate in $\mathbb{R}^n$ uses affine functions, but on manifolds, hyperplanes are replaced with paraboloids interpreted as graphs of the distance squared relative to a point. If a lower bound condition on the curvature is added, this method has been successfully used to prove geometric inequalities (see for example, \cite{ABP1, ABP2}). In particular, Brendle \cite{BrendleF} used this method to prove a monotonicity result in terms of the Jacobian of a certain transport map, while Wang \cite{wang22} and Ma-Wu \cite{ma-wu} independently generalized this method to $k$-intermediate Ricci curvature. It will be a key ingredient in our proof.

\paragraph{Acknowledgements:}
 The authors gratefully acknowledge the valuable insights and guidance provided by their advisor Guofang Wei through helpful discussions. We would like to thank Frank Morgan for the references \cite{Stam} and \cite{Federbush}. We would also like to acknowledge Lingen Lu for bringing the manuscript \cite{DongLinLu2} to our attention. Their contribution allowed us to simplify Section 2 in the first version of our work. And thanks Jing Wu for reaching us with the updated version of their paper \cite{ma-wu} and useful comments.

\section{Asymptotic Estimate }
In this section we recall an estimate for $\theta_h$, which will be used in the final steps of the proof in the next section. We first recall:

\begin{lemma}[Lemma 2.1, \cite{Yi-Zheng}]\label{lem:theta-limitASYMP2}
    Let $M$ be a complete non-compact Riemannian manifold of dimension $N$ with non-negative Ricci curvature. Then,
    \begin{equation*}
        \theta= \lim_{r\to \infty} \left( \frac{(4 \pi )^{-\frac{N}{2}}}{r^N}\int_M e^{-\frac{d(x,p)^2}{4r^2}} d V (x)\right)
    \end{equation*}
    for any $p \in M$.
\end{lemma}
The proof is based on the Bishop-Gromov volume comparison. However, in our asymptotic case, we utilize the inequality $r\leq h(r)$ instead. The upper bound for $h$, as provided in \cite{Zhu}, is $h(r) \leq e^{b_0}r$. Generally, the presence of $b_0\neq 0$ precludes us from replacing $r^N$ with $h^N$ in the above lemma. Following Dong-Lin-Lu \cite{DongLinLu2}, it is preferable to replace $r^N$ with $rh^{N-1}$, doing so, we first recall the definition of $P(t)$:
\begin{align*}
    P(t)=(4 \pi )^{-\frac{N}{2}}\int_{\mathbb{R}^N} e^{-\frac{(|x|+t)^2}{4}}dx.
\end{align*}
\begin{lemma}[Lemma 2.2, \cite{DongLinLu2}]\label{lem:theta-limitASYMP3}
    Let $M$ be a complete non-compact Riemannian manifold of dimension $N$ with asymptotic non-negative Ricci curvature. Then
    \begin{equation}\label{eq:asymptest}
         P(t)\theta_h =\lim_{r \to \infty}\frac{(4 \pi )^{-\frac{N}{2}}}{rh(r)^{N-1}} \int_M e^{- \frac{\left(\frac{d(x,o)}{r}+t\right)^2}{4}} dV (x), 
    %$$\lim_{r\to \infty}(4 \pi)^{-\frac{n}{2}}\frac{\int_M e^{-\frac{\frac{d(x,o)}{r}+t}{4}} d V (x)}{n\int_0^rh^{k-1}(t)dt} = \theta_h
    \end{equation}
    where $o$ is the base point of $M$.
\end{lemma}
This was proved in  \cite{DongLinLu2}, we are giving the same proof of this statement also here for completeness:
\begin{proof}
    First notice that we can apply De L'Hopital's rule to the limit of $\theta_h$ and obtain
    \begin{align*}
        \theta_h=\lim_{r\rightarrow\infty}\frac{\mathcal|B_r(o)|}{N\omega_N\int_0^rh^{N-1}(t)dt}=\lim_{r\rightarrow\infty}\frac{|\partial B_r(o)|}{N\omega_Nh^{N-1}(r)}
    \end{align*}
Now we can compute
\begin{align*}
    \lim_{r \to \infty}\frac{(4 \pi )^{-\frac{N}{2}}}{rh(r)^{N-1}} \int_M e^{- \frac{\left(\frac{d(x,o)}{r}+t\right)^2}{4}} dV (x)&= \lim_{r \to \infty}(4 \pi )^{-\frac{N}{2}}\int_0^\infty\frac{|\partial B_{rs}(o)|}{h^{N-1}(r)}e^{-\frac{(s+t)^2}{4}}ds\\
     &= \lim_{r \to \infty}(4 \pi )^{-\frac{N}{2}}\int_0^\infty\frac{|\partial B_rs(o)|}{h^{N-1}(rs)} \frac{h^{N-1}(rs)}{h^{N-1}(r)}  e^{-\frac{(s+t)^2}{4}}ds\\
     &=(4 \pi )^{-\frac{N}{2}}N \omega_N \theta \int_0^\infty s^{N-1}e^{-\frac{(s+t)^2}{4}}ds\\
     &= \theta (4 \pi )^{-\frac{N}{2}}\int_{\mathbb{R}^n}e^{-\frac{(|x|+t)^2}{4}}dx\\
     &=P(t)\theta,
\end{align*}
    where we used the monotone convergence theorem when we pass the limit under the integral sign.
\end{proof}
For the next lemma, we aim to substitute the base point $o$ on the right-hand side of inequality \eqref{eq:asymptest} with $\psi(x)$, where $\psi$ is a Borel map whose image is compact. 
\begin{lemma}\label{lem:theta-limitASYMP4}
    Let $M$ be a complete non-compact Riemannian manifold of dimension $n$ with asymptotic non-negative intermediate Ricci curvature.
    Then we have
    $$\lim_{r \to \infty}\frac{(4 \pi )^{-\frac{N}{2}}}{rh(r)^{N-1}}  \int_M e^{- \frac{\left(\frac{d(x,\psi(x))}{r}+t\right)^2}{4}} dV (x)=\lim_{r \to \infty}\frac{(4 \pi )^{-\frac{N}{2}}}{rh(r)^{N-1}}  \int_M e^{- \frac{\left(\frac{d(x,o)}{r}+t\right)^2}{4}} dV (x),$$
    where $\psi:M\to K$ is any Borel map where $K\subset M$ a compact subset.
\end{lemma}

The above lemma is proven, similar to the proof of Lemma 2.2 in \cite{Yi-Zheng}, by utilizing the triangle inequality. 

\section{Proof of Theorem \ref{thm:main-logASYMP}}
In this section, we prove log Sobolev inequality with non-negative asymptotic intermediate Ricci curvature by following the papers 
\cite{dong-lin-lu,DongLinLu2,Yi-Zheng}.

Let $M^{n+m}$ be a complete non-compact manifold of dimension $n+m$ with asymptotically non-negative $\mathrm{Ric}_{k}$ and Euclidean volume growth, where $k= \min \{n-1, m-1\}$. Assume that $\Sigma^n$ is an $n$-dimensional compact submanifold without a boundary and let $f$ be any positive smooth function.

We first assume that $\Sigma$ is connected, which is needed for the existence of a solution to a differential equation.
The inequality \eqref{ineq:log-sobolevASYMP} is invariant when we scale a function $f$. Thus, by scaling, we may assume that
\begin{equation}\label{eq:scalingAsymp}
    \int_\Sigma f \log f dV  - \int_\Sigma \frac{|\nablas f|^2}{f} dV - \int_\Sigma f |H|^2 d V = 0.
\end{equation}
Indeed, we can scale the function $f$ to satisfy the equation \eqref{eq:scalingAsymp} because $\log (cf)$ can be any real number by choosing a suitable constant $c$.
The left-hand-side of equation \eqref{eq:scalingAsymp} comes from the inequality \eqref{ineq:log-sobolevASYMP}.
Thus, it suffices to show that
\begin{align*}
    \int_\Sigma &f\left( \log P(4b_1 n ) \theta_h + \frac{n}{2}\log (4 \pi) +n  + 4b_1^2 n^2 + (n+m-1) \log \frac{1+b_0}{e^{2r_0b_1 + b_0}}\right)dV\\
    &\leq \left(\int_\Sigma f dV\right)\log \left(\int_\Sigma f dV\right).
\end{align*}

In order to prove Isoperimetric or Sobolev inequalities using the ABP method (see \cite{Cabre-survey} and related literature), we need to consider a suitable PDE on a submanifold $\Sigma$. To apply our assumption \eqref{eq:scalingAsymp}, we consider a differential equation as follows:
\begin{equation}\label{eq:AsympPDE}
\mathrm{div}_\Sigma (f \nablas u) = f \log f - \frac{|\nablas f |^2 }{f} - f | H|^2 \mbox{ on } \Sigma,
\end{equation}
where $\nabla^\Sigma$ is the induced Levi-Civita connection on $\Sigma$.
We do not need t 
 o specify a boundary condition since  $\partial \Sigma = \emptyset$.
Using standard PDE theory (see \cite{Trud} Chapter 6), since $f$ is a positive function, the operator $u \mapsto \mathrm{div}_\Sigma (f \nablas u)$ has Fredholm index $0$. Hence, we can find a smooth solution $u: \Sigma \to \RR$.
%CHANGETHISWe can easily verify that the differential equation \eqref{eq:AsympPDE} has a weak solution, and by standard Schauder theory (see \cite{Trud} Chapter 6), we can conclude that $u \in C^{2,\alpha}(\Sigma)$. This regularity is necessary for our proof, as we will define a map $\Phi$ containing a term $\nabla^\Sigma u$, and then compute $D\Phi$.

We define a contact set $A_r$ as the set of all points $(\bar x , \bar y) \in T^\bot\Sigma$ satisfying
\begin{align}\label{ineq:asymp-contact-set}
    r u(x) + \frac{1}{2} d(x, \exp_{\bar x} ( r \nablas u (\bar x) +  r\bar y))^2 \geq  ru (\bar x) + \frac{1}{2} r^2(|\nablas u(\bar x )|^2 + |\bar y|^2)
\end{align}
for all $x \in \Sigma$ as  introduced in \cite{BrendleF}. This is a natural extension of the Cabrè's idea \cite{Cabre} to the submanifold case.

Let $\Phi_t : T^\bot \Sigma \to M$ be a map defined by
$$\Phi_t (x,y) = \exp_x (t \nablas u (x) + t y).$$
Then we have the following lemma, for which we will recall the proof for the reader's convenience.

\begin{lemma}[Lemma 3.2, \cite{Yi-Zheng}]\label{lem:asymp-whole-cover}
For each $r \in (0, \infty)$, we have $\Phi_r (A_r) = M$.
\end{lemma}

\begin{proof}
Take any $p \in M$. Since $\Sigma$ is compact, there is $\bar x \in \Sigma$ where the function $x \mapsto  ru(x) + \frac{1}{2} d(x,p)^2$ attains its minimum.
Let $\bar \gamma$ be a minimizing geodesic such that $\bar \gamma(0) = \bar x$ and $\bar \gamma (r) = p$.
Then the geodesic $\bar \gamma$ minimizes the functional $u(\gamma(0)) + \frac{1}{2} \int_0^r |\gamma'(t)|^2 dt$ among all smooth curves $\gamma$ with  $\gamma(0) \in \Sigma$ and $\gamma(r) =p$. 
By the first variational formula, we have
$$\nablas u( \bar x) - \bar \gamma'(0) \in T_{\bar x}^\bot \Sigma.$$
That is, there is $\bar y \in T_{x}^\bot \Sigma$ satisfying
$$\nablas u (\bar x ) + \bar y = \bar \gamma'(0).$$
We can easily check that $\Phi_r(\bar{x}, \bar{y}) = p$ and $(\bar{x}, \bar{y}) \in A_r$ using the geodesic $\bar \gamma$, which implies that $p \in \Phi_r(A_r)$.
\end{proof}

As we see in the above proof, we can define a map $\psi:M \to \Sigma$ as $\psi(p)$ is the point where the function $ru(x) + \frac{1}{2} d(x,p)^2$ attains its minimum. Then  by Lemma \ref{lem:asymp-whole-cover}, we have
\begin{align}\label{ineq:asymp-temp1}
    \int_M e^{-\left( \frac{d (\psi(p),p)}{2r}+2b_1n\right)^2} \, dV (p) 
    &\leq \int_M \left(\int_{\{(x,y)\in A_r|\Phi_r(x,y) = p\}} e^{-\left(\frac{d(x, \Phi_r(x,y))}{2r} + 2b_1n \right)^2}d \mathcal{H}^0\right) d V (p),
\end{align}
where $\mathcal{H}^0$ is the counting measure.
The left hand side of the inequality \eqref{ineq:asymp-temp1} is related to the asymptotic volume ratio $\theta_h$ as we showed in Lemma \ref{lem:theta-limitASYMP3}.
Also, we apply Area formula with the map $\Phi_r: A_r \to M$ to the right hand side of \eqref{ineq:asymp-temp1}. That is,
\begin{align}\label{ineq:asymp-temp2}
    \int_M &\left(\int_{\{(x,y)\in A_r|\Phi_r(x,y) = p\}} e^{-\left(\frac{d(x, \Phi_r(x,y))}{2r} + 2b_1n \right)^2}d \mathcal{H}^0\right) d V (p)\nonumber\\
    &= \int_{A_r} e^{-\left(\frac{d(x, \Phi_r(x,y))}{2r} + 2b_1n \right)^2} |\mathrm{det} D\Phi_r(x,y)| d V (x,y)\nonumber\\
    & = \int_\Sigma \left(\int_{T_x^\bot \Sigma} e^{-\left(\frac{d(x, \Phi_r(x,y))}{2r} + 2b_1n \right)^2} |\mathrm{det} D\Phi_r(x,y)| \chi_{A_r}(x,y) \, dy\,\right) d V (x).
\end{align}

We now want to estimate the integrand in the right hand side of \eqref{ineq:asymp-temp2}. Take any $ r >0$ and  $(\bar{x},\bar y )$ in $A_r$. Define a geodesic $\bar \gamma(t) = \exp_{\bar x}(t \nablas u(\bar x) + t \bar y)$ on $[0,r]$. Let $\{e_1,\ldots, e_n\}$ be an orthonormal basis of $T_{\bar x}\Sigma$ such that $\mathrm{Hess}_\Sigma u (e_i, e_j) - \langle II(e_i, e_j), \bar y\rangle$ is diangonal. Let us denote $E_i$ by a parallel transport vector field of $e_i$ along $\bar \gamma$ for each $1 \leq i\leq n$.
Take any orthonormal frame $\{e_{n+1},\ldots, e_{n+m}\}$ of $T^\bot \Sigma$ near $\bar x$ such that $\langle \nabla_{e_i} e_\alpha , e_\beta \rangle = 0$ at $\bar x$ for all $ 1 \leq i \leq n$ and $ n+1 \leq \alpha, \beta \leq n+m$.

For each $1 \leq i \leq n$, let us consider a Jacobi field $X_i(t)$ along $\bar \gamma(t)$ with the initial conditions
$$\begin{cases}
    X_i(0) = e_i\\
    \langle X_i' (0) , e_j \rangle = \mathrm{Hess}_\Sigma u (e_i, e_j) - \langle II(e_i, e_j) , \bar y \rangle \quad \mbox{ for all } 1 \leq j \leq n\\
    \langle X_i'(0) , e_\alpha \rangle = \langle II (e_i, \nablas u (\bar x )) , e_\alpha \rangle
    \quad \mbox{ for all } n+1 \leq \alpha \leq n+m.
\end{cases}$$
For each $n+1 \leq \alpha \leq n+m$, consider a Jacobi field $X_\alpha (t)$ along $\bar \gamma(t)$ with the initial conditions
$$
    X_\alpha(0) = 0 \quad \mbox{and} \quad
    X_\alpha' (0) = e_\alpha.
$$

Define $(n+m)\times (n+m)$ matrices $P(t)$ and $S(t)$ satisfying $P_{ij}(t) = \langle X_i (t) , E_j(t) \rangle$ and
$S_{ij}(t) = \bar R (\bar \gamma '(t) , E_i(t),\bar \gamma'(t),E_j(t))$ for $ 1 \leq i, j \leq n+m$.
Since $X_i(t)$ is a Jacobi field along $\bar \gamma (t)$, we have $P''(t)= - P(t) S (t)$.
Define $Q(t) = P^{-1}(t)P'(t)$. Then 
$$Q'(t) = - (Q(t))^2 -S(t).$$
Let $tr_n Q(t) = \sum_{i=1}^{n}Q_{ii}(t)$ and $tr_m Q(t) = \sum_{\alpha=n+1}^{n+m} Q_{\alpha \alpha}(t)$.
Then
$$(tr_n Q(t))' + \frac{1}{n} (tr_n Q(t))^2 \leq - tr_n S(t)$$
and
$$(tr_m Q(t))' + \frac{1}{m} (tr_m Q(t))^2 \leq - tr_m S(t),$$
where $tr_n S(t)$ and $tr_m S(t)$ are defined as a partial trace of $S(t)$.
Let us note that
$$tr_nS(t) = \sum_{i=1}^n \bar R (\bar \gamma'(t), E_i(t), \bar \gamma'(t), E_i(t)).$$
However, we cannot apply curvature assumption directly because $\bar \gamma'(t)$ may not be perpendicular to the plane spanned by $E_1(t), \ldots, E_n(t)$. Since $E_i(t)$ is a parallel vector field along $\bar \gamma(t),$ it is enough to consider the angle at $t =0$, denoted by $a$,  between the vector $\bar \gamma'(0)$ and the tangent plane $T_{\bar x } \Sigma$.
Since $\bar \gamma'(0) = \nablas u (\bar x ) + \bar y$, depending on the vectors $\nablas u (\bar x )$ and $ \bar y$ the angle is determined.

Let us consider the case $\nablas u ( \bar x ) \neq 0$ and $\bar y \neq 0$.
In this case, $\bar \gamma'(0)$ is not perpendicular to the tangent plane. After we change the basis as in \cite{ma-wu}, we get
    \begin{align}\label{ineq:trn}
        tr_n S(t)
        & = \sin^2 (a)|\bar \gamma'(0)|^2 Ric_n\left(\frac{\bar \gamma'(t)}{|\bar \gamma'(t)|}, P_1(t)\right) + \cos^2(a)|\bar \gamma'(0)|^2
        Ric_{n-1}\left(\frac{\bar \gamma'(t)}{|\bar \gamma'(t)|}, P_2(t)\right)\nonumber\\
        & \geq -n\sin^2 (a) |\bar \gamma '(0)|^2  \lambda(d(o, \bar \gamma(t))) - (n-1) \cos^2 (a) |\bar \gamma'(0)|^2 \lambda (d( o, \bar \gamma(t)))\nonumber\\
        &=  (\cos^2 ( a) - n)  |\bar \gamma'(0)|^2\lambda(d(o, \bar \gamma(t))),
    \end{align}
    where $P_1(t)$ is an $n$-dimensional subspace spanned by $E_2(t), \ldots, E_{n+1}(t)$ and $P_2(t)$ is an $(n-1)$-dimensional subspace spanned by $E_2(t),\ldots, E_n(t)$.
    Similarly, we have
\begin{align}\label{ineq:trm}
    tr_mS(t)  \geq (\sin^2 (a) - m) |\bar \gamma'(0)|^2 \lambda (d(o,\bar \gamma(t))).
\end{align}
Indeed, the inequalities \eqref{ineq:trn} and \eqref{ineq:trm} hold for the case $\nablas u (\bar x) =0$ or $\bar y =0$. More specifically, for the case $\nablas u (\bar x) = 0$,
since $\bar \gamma' ( 0 ) = \bar y \in T_{\bar x}^\bot\Sigma$, the angle $a = \frac{\pi}{2}$ and we have
\begin{align*}
    tr_n S(t) 
    &= |\bar \gamma'(0)|^2Ric_n(P_3(t), \bar \gamma'(t))\\
    &\geq -n |\bar \gamma'(0)|^2\lambda (d(o, \bar \gamma(t))) \\
    &= (\cos^2 ( a) - n)|\bar \gamma'(0)|^2   \lambda(d(o, \bar \gamma(t)))
\end{align*}
and
\begin{align*}
    tr_m S(t) &= \sum_{\alpha = n+1}^{n+m} \bar R (\bar \gamma'(t) , E_\alpha(t), \bar \gamma'(t) , E_\alpha(t))\\
    &= \sum_{\alpha=n+2}^{n+m} \bar R (\bar \gamma'(t) , E_\alpha(t), \bar \gamma'(t) , E_\alpha(t)) \\
    &= |\bar \gamma'(0)|^2 Ric_{m-1}(P_4 (t), \bar \gamma'(t)) \\
    &\geq -(m-1) |\bar \gamma'(0)|^2\lambda (d(o, \bar \gamma'(t)))\\
    & \geq (\sin^2 (a) - m) |\bar \gamma'(0)|^2 \lambda (d(o,\bar \gamma(t))),
\end{align*}
where $P_3(t) = Span (E_1(t), \ldots, E_n(t))$ and $P_4 (t) = Span (E_{n+2}(t), \ldots, E_{n+m}(t))$
with the assumption $E_{n+1} = \frac{\bar y}{|\bar y|}$. 
In the same manner, for the case $\bar y = 0$, since
$\bar \gamma'(0) = \nablas u(\bar x) \in T_{\bar x} \Sigma$, by assuming $E_{1}(t) = \frac{\bar \gamma'(t)}{|\bar \gamma'(t)|}$, we can verify that the same inequalities \eqref{ineq:trn} and \eqref{ineq:trm} hold.
In all, for $(\bar x, \bar y) \in A_r \setminus \{(0,0)\}$, we obtain
$$
\begin{cases}
    (tr_nQ(t))' + \frac{1}{n}(tr_nQ(t))^2  \leq (n-\cos^2 (a) )  |\bar \gamma'(0)|^2\lambda (d(o,\bar \gamma(t)))\\
    (tr_mQ(t))' + \frac{1}{m}(tr_mQ(t))^2  \leq (m-\sin^2 (a))|\bar \gamma'(0)|^2\lambda (d(o,\bar \gamma(t)))
\end{cases}$$
by combining inequalities \eqref{ineq:trn}  and \eqref{ineq:trm} with Riccati equation.
It follows from the triangle inequality
$d(o, \bar \gamma(t)) \geq |d(o, \bar x ) - d( \bar x , \bar \gamma(t))|$ and the definition of $A_r$ that
$$\begin{cases}
    (tr_nQ(t))' + \frac{1}{n}(tr_nQ(t))^2  \leq (n-\cos^2(a)) |\bar \gamma'(0)|^2
    \lambda\left(\left|d ( o, \bar x ) - t \sqrt{|\nablas u ( \bar x ) |^2 + |\bar y |^2}\right|\right)\\
    (tr_mQ(t))' + \frac{1}{m}(tr_mQ(t))^2  \leq (m-\sin^2(a))|\bar \gamma'(0)|^2\lambda\left(\left|d ( o, \bar x ) - t \sqrt{|\nablas u ( \bar x ) |^2 + |\bar y |^2}\right|\right).
\end{cases}$$

Let $\phi(t) = e^{\frac{1}{n}\int_0^t \sum_{i=1}^n Q_{ii}(\tau ) d \tau}$.
    Then
     $$
    \begin{cases}
        \phi''(t) \leq \frac{n-\cos^2 (a)}{n}|\bar \gamma'(0)|^2\lambda\left(\left|d ( o, \bar x ) - t \sqrt{|\nablas u ( \bar x ) |^2 + |\bar y |^2}\right|\right) \,\phi(t)\\
        \phi'(0) = \frac{1}{n} ( \Delta_{\Sigma} u (\bar x) - \langle H (\bar x) , \bar y \rangle )\\
        \phi(0) =1.
    \end{cases}
    $$
To use comparison theorems with $\phi (t)$, let us define $\psi_1 (t)$ and $\psi_2(t)$ as a solution to each PDE similarly to \cite{dong-lin-lu}:
$$\begin{cases}
    \psi_1''(t) =  \frac{n-\cos^2 (a)}{n}|\bar \gamma'(0)|^2\lambda\left(\left|d ( o, \bar x ) - t \sqrt{|\nablas u ( \bar x ) |^2 + |\bar y |^2}\right|\right) \,\psi_1(t)\\
    \psi_1(0) = 0,  \psi_1'(0) = 1
\end{cases}
$$
and
$$
\begin{cases}
    \psi_2''(t) =  \frac{n-\cos^2 (a)}{n}|\bar \gamma'(0)|^2\lambda\left(\left|d ( o, \bar x ) - t \sqrt{|\nablas u ( \bar x ) |^2 + |\bar y |^2}\right|\right) \,\psi_2(t)\\
    \psi_2(0) = 1,  \psi_2'(0) = 0.
\end{cases}$$
Then we see that a function $\psi(t) = \psi_2(t) + \frac{1}{n} ( \Delta_{\Sigma} u (\bar x) - \langle H (\bar x) , \bar y \rangle ) \psi_1(t)$ satisfies the following PDE.
$$
    \begin{cases}
        \psi''(t) = \frac{n-\cos^2 (a)}{n}|\bar \gamma'(0)|^2\lambda\left(\left|d ( o, \bar x ) - t \sqrt{|\nablas u ( \bar x ) |^2 + |\bar y |^2}\right|\right) \,\psi(t)\\
        \psi'(0) =  \frac{1}{n} ( \Delta_{\Sigma} u (\bar x) - \langle H (\bar x) , \bar y \rangle )\\
        \psi(0) =1.
    \end{cases}
    $$
By comparing $\frac{\phi'(t)}{\phi(t)}$ with $\frac{\psi'(t)}{\psi(t)}$, we have
    $$tr_n Q(t) = n \cdot \frac{\phi'(t)}{\phi(t)} \leq n \cdot \frac{\psi'(t)}{\psi(t)} = \frac{\psi_2' +\frac{1}{n} ( \Delta_{\Sigma} u (\bar x) - \langle H (\bar x) , \bar y \rangle )  \psi_1' }{\psi_2 + \frac{1}{n} ( \Delta_{\Sigma} u (\bar x) - \langle H (\bar x) , \bar y \rangle )  \psi_1}.$$
Define $$\bar \phi(t) = t e^{\frac{1}{m} \int_0^t \sum_{\alpha =n+1}^{n+m} (Q_{\alpha \alpha}(t) - \frac{1}{t} ) d\tau}.$$
Then
$$\begin{cases}
    \bar \phi'' \leq \frac{m- \sin^2 (a)}{m}|\bar \gamma'(0)|^2\lambda\left(\left|d(o,\bar x ) - t \sqrt{|\nablas u(\bar x)|^2 + |\bar y |^2})\right|\right) \, \bar \phi\\
    \bar \phi(0) = 0, \quad 
    \bar \phi'(0) = 1.
\end{cases}$$
To estimate $\frac{\bar \phi'}{\bar \phi}$, let us introduce a $C^2$ solution $\bar \psi$ to the PDE
$$\begin{cases}
    \bar \psi'' = \frac{m- \sin^2 (a)}{m}|\bar \gamma'(0)|^2\lambda\left(\left|d(o,\bar x ) - t \sqrt{|\nablas u(\bar x)|^2 + |\bar y |^2})\right|\right) \, \bar \psi\\
    \bar \psi(0) = 0, \quad 
    \bar \psi'(0) = 1.
\end{cases}$$
Thus,
$$tr_mQ(t) = m \cdot \frac{\bar \phi'}{\bar \phi} \leq m \cdot \frac{\bar \psi'}{\bar \psi}.$$

Since
$\frac{d}{dt} (\det P(t)) = \det P(t) \tr Q(t),$
we have
$$\frac{d}{dt} \log (\det P(t)) = tr Q(t) \leq n\cdot \frac{\psi'}{\psi } + m \cdot \frac{\bar \psi' }{\bar \psi} = \frac{d}{dt} (n\log (\psi) + m \log (\bar \psi)) = \frac{d}{dt} \log (\psi^n \bar \psi^m).$$

By integrating the above inequality from $\epsilon$ to $t$ and $\epsilon \to 0$, we get
\begin{align*}
    \det P(t) \leq \psi^n(t) \bar \psi^m(t) &= \left(\psi_2(t) + \frac{1}{n} ( \Delta_{\Sigma} u (\bar x) - \langle H (\bar x) , \bar y \rangle ) \psi_1(t)\right)^n \bar \psi^m(t)\\
    &= \left(\frac{\psi_2(t)}{\psi_1(t)} + \frac{1}{n} (\Delta_{\Sigma} u (\bar x) - \langle H (\bar x) , \bar y \rangle )  \right)^n \psi_1^n(t) \bar \psi^m (t)
\end{align*}
Moreover, comparison theorems imply those inequalities
$$\begin{cases}
    \frac{\psi_2}{\psi_1} \leq \frac{2b_1(n-\cos^2(a))}{n} \sqrt{|\nablas u ( \bar x ) |^2 + |\bar y |^2}  + \frac{1}{t}\\
    \psi_1(t) 
    \leq  t e^{\frac{n- \cos^2 (a)}{n}(2b_1 d(o, \bar x ) + b_0)}\\
    \bar \psi (t) \leq t e^{\frac{m- \sin^2 (a)}{m}(2b_1 d(o, \bar x ) + b_0)}.
\end{cases}$$
Thus,
\begin{align}\label{lem:asymp-estimate}
    &\det P(t) \nonumber\\
    &\leq \left(\frac{\psi_2(t)}{\psi_1(t)} + \frac{1}{n} (\Delta_{\Sigma} u (\bar x) - \langle H (\bar x) , \bar y \rangle )  \right)^n \psi_1^n(t) \bar \psi^m (t)\nonumber\\
    &\leq \left(\frac{2b_1(n - \cos^2 (a))}{n} \sqrt{|\nablas u ( \bar x ) |^2 + |\bar y |^2}  + \frac{1}{t}+ \frac{1}{n} (\Delta_{\Sigma} u (\bar x) - \langle H (\bar x) , \bar y \rangle )  \right)^n \nonumber\\
    & \quad \cdot t^n e^{(n- \cos^2 (a))(2b_1 d(o, \bar x ) + b_0)} \cdot t^m e^{(m- \sin^2 (a))(2b_1 d(o, \bar x ) + b_0)}\nonumber\\
    & \leq \left(2b_1 \sqrt{|\nablas u ( \bar x ) |^2 + |\bar y |^2}  + \frac{1}{t}+ \frac{1}{n} (\Delta_{\Sigma} u (\bar x) - \langle H (\bar x) , \bar y \rangle )  \right)^n t^{n+m}e^{(n+m-1)(2b_1 d( o,\bar x)+b_0)}\nonumber\\
    &\leq \left(2b_1\sqrt{|\nablas u ( \bar x ) |^2 + |\bar y |^2}  + \frac{1}{t}+ \frac{1}{n} (\Delta_{\Sigma} u (\bar x) - \langle H (\bar x) , \bar y \rangle )  \right)^n t^{n+m}e^{(n+m-1)(2b_1 r_0+b_0)},
\end{align}

The following lemma can be also found in the proof of Lemma 3.7 in \cite{Yi-Zheng}.
\begin{lemma}\label{lem:asymp:laplacian} 
For $(x,y) \in A_r$, we have:
\begin{equation*}\label{asymplap}
    \Delta_\Sigma u (x) - \langle H(x) , y \rangle\leq \log f(x)+ \frac{|\nabla^\Sigma u|^2 + |y|^2}{4}- \frac{|2H(x)+y|^2}{4}.
\end{equation*}
\end{lemma}

\begin{proof}
Since $\mathrm{div}_\Sigma (f \nablas u) = f \log f - \frac{|\nablas f |^2 }{f} - f | H|^2$, we get
$$f \Delta_\Sigma u + \langle \nablas f, \nablas u \rangle  = f \log f - \frac{|\nablas f |^2 }{f} - f | H|^2.$$
Since $f(x) \neq 0$ for all $x \in \Sigma$, we have
\begin{align}\label{eq:temp-1}
    \Delta_\Sigma u - \langle H(x) , y \rangle 
     =  \log f - \frac{|\nablas f |^2 }{f^2} -  | H(x)|^2 -\frac{\langle \nablas f , \nablas u \rangle}{f} - \langle H(x) , y\rangle.
\end{align}
Note that 
\begin{align}\label{eq:temp-2}
    \frac{|\nablas f |^2}{f^2} + \frac{\langle \nablas f , \nablas u \rangle}{f} & = \frac{4 |\nablas f |^2 + 4 \langle \nablas f, f \nablas u \rangle }{4f^2} \nonumber\\
    & = \frac{4 |\nablas f|^2 + 4 \langle \nablas f, f \nablas u \rangle + |f \nablas u |^2}{4f^2} -\frac{|f \nablas u|^2}{4f^2}\nonumber\\
    & = \frac{| 2 \nablas f+ f \nablas u |^2 }{4f^2} - \frac{|\nablas u|^2}{4}
\end{align}
and
\begin{align}\label{eq:temp-3}
    |H(x)|^2 + \langle H(x), y\rangle & = \frac{|2H(x)|^2 + \langle 4H(x), y \rangle + |y|^2}{4} - \frac{|y|^2}{4} = \frac{|2H(x) + y|^2}{4} - \frac{|y|^2}{4}.
\end{align}
Combining identities \eqref{eq:temp-1}, \eqref{eq:temp-2}, and \eqref{eq:temp-3}, we obtain
\begin{align*}
    \Delta_\Sigma u (x) &- \langle H(x) , y \rangle 
    \\& =  \log f - \frac{|\nablas f |^2 }{f^2} -  | H|^2 -\frac{\langle \nablas f , \nablas u \rangle}{f} - \langle H(x) , y\rangle \\
    & = \log f + \frac{|\nablas u (x)|^2 + |y|^2}{4} - \frac{|2 \nablas f + f \nablas u |^2}{ 4f^2} - \frac{|2H(x) + y |^2}{4}\\
    & \leq \log f + \frac{|\nablas u(x) |^2 + |y|^2}{4} - \frac{|2H(x)+y|^2}{4}.
\end{align*}
\end{proof}

\begin{lemma}\label{lem:estimate-expASYMP}
For $(\bar x,\bar y) \in A_r \setminus  \{(0,0)\}$, we have:
$$|\mathrm{det} D\Phi_t (\bar x,\bar y)| \leq e^{\alpha + \frac{n}{t} -n - 4b_1^2 n^2} t^{m+n} f\exp \left(\left(\frac{d(\bar x, \Phi_t(\bar x,\bar y))}{2t} + 2b_1n\right)^2\right) \exp \left(- \frac{|2H+\bar y|^2}{4} \right),$$
where $\alpha = (n+m-1)(2r_0b_1+b_0)$.
\end{lemma}
\begin{proof}
    Combining inequality \eqref{lem:asymp-estimate} and Lemma \ref{lem:asymp:laplacian}, we obtain
    \begin{align*}
         &\det P(t) \\&\leq \left(2b_1 \sqrt{|\nablas u ( \bar x ) |^2 + |\bar y |^2}  + \frac{1}{t}+ \frac{1}{n} (\Delta_{\Sigma} u (\bar x) - \langle H (\bar x) , \bar y \rangle )  \right)^n t^{n+m}e^{\alpha}\\
         & \leq \left(2b_1\sqrt{|\nablas u ( \bar x ) |^2 + |\bar y |^2}  + \frac{1}{t}+ \frac{1}{n} \left(\log f+ \frac{|\nabla^\Sigma u|^2 + |\bar y|^2}{4}- \frac{|2H(\bar x)+\bar y|^2}{4} \right)  \right)^n t^{n+m}e^{\alpha}
    \end{align*}
By the inequality $ \lambda \leq e^{\lambda -1}$, we have
\begin{align*}
    &|\mathrm{det} D\Phi_t (\bar x,\bar y)| \\
    &\leq\exp \left(2b_1n \sqrt{|\nablas u ( \bar x ) |^2 + |\bar y |^2}  + \frac{n}{t}+ \log f+ \frac{|\nabla^\Sigma u|^2 + |\bar y|^2}{4}- \frac{|2H+\bar y|^2}{4} -n\right) t^{n+m}e^{\alpha}\\
    & =e^{\alpha + \frac{n}{t} -n} t^{m+n} f\exp \left(2b_1n \sqrt{|\nablas u ( \bar x ) |^2 + |\bar y |^2}  + \frac{|\nabla^\Sigma u|^2 + |\bar y|^2}{4}- \frac{|2H+\bar y|^2}{4}\right)\\
    & \leq e^{\alpha + \frac{n}{t} -n} t^{m+n} f(\bar x)\exp \left(2b_1n \sqrt{|\nablas u ( \bar x ) |^2 + |\bar y |^2}  + \frac{|\nabla^\Sigma u|^2 + |\bar y|^2}{4}- \frac{|2H(\bar x)+\bar y|^2}{4}\right)\\
    & = e^{\alpha + \frac{n}{t} -n} t^{m+n} f(\bar x)\exp \left(\left(\frac{\sqrt{|\nabla^\Sigma u|^2 + |\bar y|^2}}{2} + 2b_1n\right)^2\right) \exp \left(- \frac{|2H+\bar y|^2}{4} - 4 b_1^2 n^2\right)\\
    & = e^{\alpha + \frac{n}{t} -n-4b_1^2 n^2} t^{m+n} f(\bar x)\exp \left(\left(\frac{d(\bar x, \Phi_t(\bar x,\bar y))}{2t} + 2b_1n\right)^2\right) \exp \left(- \frac{|2H(\bar x)+\bar y|^2}{4} \right).
\end{align*}

\end{proof}

By combining inequalities \eqref{ineq:asymp-temp1} and \eqref{ineq:asymp-temp2} and Lemma \ref{lem:estimate-expASYMP}, we obtain
\begin{align*}
    \int_M &e^{- \left(\frac{d(\psi (p), p)}{2t} + 2b_1n\right)^2} \, d V (p) \\
    &\leq \int_\Sigma \left(\int_{T_x^\bot \Sigma} e^{-\left(\frac{d( x, \Phi_t(x, y))}{2t} + 2b_1n\right)^2} |\mathrm{det} D\Phi_t(x,y)| \chi_{A_t}(x,y) \, dy\,\right) d V (x)\\
    &=\int_\Sigma  \left(\int_{T_x^\bot \Sigma} e^{\alpha + \frac{n}{t} -n-4b_1^2 n^2} t^{m+n} f( x) \exp \left(- \frac{|2H( x)+ y|^2}{4} \right)dy\right)d V (x)\\
    & = t^{n+m} e^{\frac{n}{t} -n+\alpha-4b_1^2 n^2}(4 \pi)^\frac{m}{2}\int_\Sigma f (x)d V (x).
\end{align*}
Then

\begin{align}\label{temp}
    \frac{(4\pi)^{-\frac{n+m}{2}}}{t (h(t))^{n+m-1}}\int_M e^{- \left(\frac{d(\psi (p), p)}{2t} + 2b_1n\right)^2} \, d V (p)  \leq \frac{(4 \pi)^{-\frac{n}{2}}t^{n+m-1}}{(h(t))^{n+m-1}}  e^{\frac{n}{t} -n+\alpha-4b_1^2 n^2}\int_\Sigma f (x)d V (x)
\end{align}

Let us take limit $t \to \infty$ to the both sides of the above inequality. Then
by Lemma \ref{lem:theta-limitASYMP3} and \ref{lem:theta-limitASYMP4}, we get
$$P(4b_1 n )\theta_h  \leq \lim_{t \to \infty}\frac{t^{n+m-1}}{(h(t))^{n+m-1}} (4 \pi)^{-\frac{n}{2}} e^{\frac{n}{t} -n+\alpha-4b_1^2 n^2}\int_\Sigma f (x)d V (x).$$
Since
$$\lim\limits_{t \to \infty} \frac{t^{n+m-1}}{(h(t))^{n+m-1}} = \left(\lim\limits_{t \to \infty} \frac{t}{h(t)}\right)^{n+m-1} = \left(\lim\limits_{t \to \infty}\frac{1}{h'(t)}\right)^{n+m-1} \leq \frac{1}{(1+b_0)^{n+m-1}},$$
we have
$$P(4b_1 n )\theta_h  \leq  (4 \pi)^{-\frac{n}{2}} e^{ -n+\alpha-4b_1^2 n^2}\frac{1}{(1+b_0)^{n+m-1}} \int_\Sigma f (x)d V (x).$$
By taking a logarithmic function, we obtain
$$\log P(4b_1 n ) \theta_h + \frac{n}{2}\log (4 \pi) +n - \alpha + 4b_1^2 n^2 + (n+m-1) \log (1+b_0)\leq \log \int_\Sigma f(x) dV (x),$$
or equivalently,
$$\log P(4b_1 n ) \theta_h + \frac{n}{2}\log (4 \pi) +n  + 4b_1^2 n^2 + (n+m-1) \log \frac{1+b_0}{e^{2r_0b_1 + b_0}}\leq \log \int_\Sigma f(x) dV (x).$$
Multiplying the inequality by $f$ and integrate it over $\Sigma$, then
\begin{align*}
    \int_\Sigma &f \left(\log P(4b_1 n ) \theta_h + \frac{n}{2}\log (4 \pi) +n  + 4b_1^2 n^2 + (n+m-1) \log \frac{1+b_0}{e^{2r_0b_1 + b_0}}\right)d V (x) \\
    &\leq \int_\Sigma f d V (x) \left(\log \int_\Sigma f(x) dV (x)\right).
\end{align*}
By adding our assumption
$$\int_\Sigma f \log f - \int_\Sigma \frac{|\nablas f|^2}{f} - \int_\Sigma f |H|^2 = 0,$$
we conclude that
\begin{align*}
    \int_\Sigma &f \left(\log f + \log P(4b_1 n ) \theta_h + \frac{n}{2}\log (4 \pi) +n  + 4b_1^2 n^2 + (n+m-1) \log \frac{1+b_0}{e^{2r_0b_1 + b_0}}\right)d V (x) \\
    &- \int_\Sigma \frac{|\nablas f|^2}{f} - \int_\Sigma f |H|^2  \leq \int_\Sigma f d V (x) \left(\log \int_\Sigma f(x) dV (x)\right).
\end{align*}

It now remains to show the case when $\Sigma$ is disconnected. We follow Brendle \cite{B_log}  and notice that
\begin{align*}
    a\log(a)+b\log(b) <a\log(a+b)+b\log(a+b)=(a+b)\log(a+b).
\end{align*}
The above simple inequality will take care of the right hand side when we apply inequality \eqref{ineq:log-sobolevASYMP} on each connected component individually. And this completes our proof of Theorem \ref{thm:main-logASYMP}.

\bibliography{ref}
\bibliographystyle{plain}
\end{document}